\documentclass[11pt]{amsart}
\usepackage{setspace}
\usepackage{enumerate}
\usepackage{amsthm}
\usepackage{amsmath,amssymb,latexsym,amscd}
\usepackage{tikz}
\makeatletter
\@namedef{subjclassname@2010}{%
  \textup{2010} Mathematics Subject Classification}
\makeatother
\theoremstyle{definition}
\newtheorem{theorem}{Theorem}[section]
\newtheorem{lemma}[theorem]{Lemma}
\newtheorem{proposition}[theorem]{Proposition}
\newtheorem{corollary}[theorem]{Corollary}
\theoremstyle{definition}
\theoremstyle{definition}
\newtheorem{definition}[theorem]{Definition}
\newtheorem{remark}[theorem]{Remark}

\usepackage{indentfirst}

\DeclareMathOperator{\im}{im}
\def\Z{\mathbb{Z}}
\def\C{\mathbb{C}}

\def\S{\mathbb{S}}

\begin{document}
\baselineskip=17pt
\title[]{Fixed Point Free Actions of Spheres and Equivariant Maps}
\author[Anju Kumari and Hemant Kumar Singh]{ Anju Kumari and Hemant Kumar Singh}

\address{{\bf Anju Kumari}, 
	Department of Mathematics, University of Delhi,
	Delhi -- 110007, India.}
\email{anjukumari0702@gmail.com}
\address{{\bf Hemant Kumar Singh}, 
	Department of Mathematics, University of Delhi,
	Delhi -- 110007, India.}
\email{hemantksingh@maths.du.ac.in}
\date{}
\thanks{This paper is supported by the Science and Engineering Research Board (Department of Science and Technology, Government of India) with reference number- EMR/2017/002192}
\begin{abstract} 
This paper generalizes the concept of index and co-index and some related results for free actions of $G=\S^0$ on a paracompact Hausdorff space which were introduced by Conner and Floyd\cite{ConnerFloyd1960}. We define the index and co-index of a finitistic free $G$-space $X$, where $G=\S^d$, $d=1$ or 3 and prove that the index of $X$ is not more than the mod 2 cohomology index of $X$. We observe that the index and co-index  of a   $(2n+1)$-sphere $\mathbb{S}^{2n+1}$(resp. $(4n+3)$-sphere $\mathbb{S}^{4n+3})$ for the action of componentwise multiplication of $G=\S^1$ (resp. $\S^3$) is $n$.

We also  determine the orbit spaces of free actions of $G=\S^3$ on a finitistic space $X$ with the mod 2 cohomology  and the rational cohomology product of spheres   $\mathbb{S}^n\times\mathbb{S}^m, 1\leq n\leq m$. The orbit spaces of circle actions on the mod 2 cohomology $X$ is also discussed. Using these  calculation, we obtain  an upper bound of the index of $X$ and the Borsuk-Ulam type results.
\end{abstract}


\maketitle

\section {Introduction} Firstly, in 1954, C. T. Yang\cite{Yang1954} defines an index for compact Hausdorff spaces with free involutions using Smith (co)homology to study mappings from spheres to euclidean spaces like Borsuk-Ulam theorem and extension of Dyson's theorem.   In 1960, Conner and Floyd\cite{ConnerFloyd1960} defines  $\mathbb{Z}_2$-index (which is Yang's B-index\cite{Yang1955}), co-index and homology index, and  also discussed the stability of index for Hausdorff spaces. In 1962\cite{ConnerFloyd1962}, they discussed the co-index of space of paths $P(\S^n)$ except for some values of $n$. In 1972, Jack Ucci\cite{Jack} discussed the co-index for remaining values.  In 1988, Fadell and Husseni\cite{Husseni1988}  introduced the ideal valued index for  free Lie group actions on paracompact spaces. After this many author generalized index in different ways.  In 1989, Stolz\cite{Stolz1989} studied the Conner and Floyd's index for real projective spaces. Volvikov\cite{Volvikov2000}(2000)  defined numerical index $i_G(X)$ using spectral sequences and ideal valued index  with filtration for a free $G$-space  using the Borel construction for  compact Lie group $G$.  Jaworowski\cite{Jaworowski2002}(2002) defines $G$-index ($G=\S^1$ or $\S^3$) with integer  coefficients similar to Conner and Floyd's\cite{ConnerFloyd1960} homology index  and  proved that the index of $(2n+1)$-lens space is  $n$. We call Jaworowski's  $G$-index with $\Z_2$ coefficient as the mod 2 cohomology $G$-index of $X$\cite{Hemant2012,JKaur2015}. Tanaka\cite{Tanaka2003,Tanaka2007} defined index and co-index for vector bundles and its properties. In 2005, Yasuhiro Hara\cite{Hara2005} studies equivariant maps between Stiefel manifolds  using index. In 2013, Satya Deo\cite{Deo2013} proved that the numerical index   for a finitistic space with  $p$-torus actions  is finite. In 2014, Benjamin  Matschke\cite{Matschke} defined ideal valued index using spectral sequences.  We generalize Conner and Floyd's\cite{ConnerFloyd1960} index and co-index and related standard results to index and co-index  for finitistic  space $X$ with free  actions of  $G=\S^d$, $d=1$ or 3.  We also prove that index of $X$ is not mare than  the mod 2 cohomology index of $X$.

 H. Hopf in 1925-26,  raised the question to classify all manifolds whose universal cover is $\S^n$. It is equivalent to determine the orbit spaces $\S^n/G$, where $G$ is a finite cyclic group.  The orbit spaces of $n$-sphere $\S^n$ with finite group actions have been studied in \cite{Livesay,Rice,Ritter1973,Rubinstein}. In 1963,  J. C. Su\cite{Su1963} computed the orbit spaces of spheres for free circle actions.  For the product of spheres $\S^n\times \S^m$, the orbit spaces of $G=\Z_p$, $p$ a prime with $\Z_p$-coefficients or $\S^1$-actions with rational coefficients are discussed in \cite{Dotzel,Ritter1974,Tao}. In 1972, Ozeki and Uchida\cite{Ozeki} determine the  orbit space of free circle action on a manifold with integral cohomology $\S^{2n+1}\times\S^{2m+1}$.    J Kaur et.al\cite{JKaur2015}(2015) studied the fixed point free $\S^3$-actions on spheres. We generalize the discussion for free actions of $G=\S^3$ on $n$-sphere $\S^n$ to a finitistic space $X$ with the mod 2 cohomology and the rational cohomology isomorphic to the product of spheres $\S^n\times \S^m,1\leq n\leq m$, and determined the cohomological structure of the orbit spaces. The orbit spaces of $X$ with free actions of $G=\S^1$ with $\Z_2 $ coefficients is also discussed. Using these  calculations, we obtain an upper bound of the index of $X$ and  Borsuk-Ulam type results.
\section {Preliminaries}

Let $G$ be a compact Lie group. Then by using Milnor's construction\cite{dale fibre bundle}, there exist a universal principal $G$-bundle $G\hookrightarrow E_G\to B_G$. If $X$ is a free $G$-space then the associated bundle $X\stackrel{i}{\hookrightarrow} X_G= \frac{X\times E_G}{G}\stackrel{\pi}{\longrightarrow} B_G $ is  a fibre bundle with fibre $G$ where $G$ acts on $X\times E_G$ by the diagonal action. This associated fibre bundle becomes a  fibration as $B_G$ is a paracompact space and this fibration is called the Borel fibration. Then, there exist Leray-Serre spectral sequence associated
 to this Borel fibration $X\hookrightarrow X_G\to B_G$ which has $E_2^{k,l}=H^k(B_G;\mathcal{H}^l(X))$,
 the cohomology of the base $B_G$ with local coefficients in the cohomology of the fibre of map $\pi$. Note that if $\pi_1(B_G)$ acts trivially on $X $ then $E_2^{k,l}=H^k(B_G)\otimes H^l(X)$.
  \begin{proposition}\label{5p2}\cite[Theorem 5.9]{Mc}
 	Let $X\stackrel{i}{\hookrightarrow} X_G \stackrel{\pi}{\longrightarrow} B_G$ be the Borel fibration. Suppose that the system of local coefficients on $B_G$ is simple, then the edge homomorphisms
 	are the homomorphisms $\pi^*: H^k(B_G) \to H^k(X_G) ~ ~ ~ \textrm{and} ~ ~ ~ i^*: H^l(X_G)  \to H^l(X).$
 \end{proposition}
	For details about the results related to the spectral sequence, we refer\cite{Mc}.

 Let $G$ be a compact Lie group which act freely on a finitistic space $X$ and $h:X_G\rightarrow X/G$ be the map induced by the $G$-equivariant projection $X\times E_G\rightarrow X$.	 Then $h$ induces an isomorphism on cohomology. Further, $X/G$ and $X_G$ have same homotopy type\cite{Dieck}.

In this paper, we have taken all actions as left actions, all spaces $X$ are assumed to be finitistic space  and $H^*(X;R)$ is notation for the  \v{C}ech cohomology with    coefficients in $R$, where $R=\mathbb{Z}_2$ or $\mathbb{Q}$. Note that $X\sim_R Y$ means $H^*(X;R)\cong H^*(Y;R)$.

The following results has also been used in this paper.
\begin{proposition}(\cite{Hemant2012,JKaur2015})\label{S1 and S3 pre}
	Let $X$ be a finitistic space with free  $G=\S^1$ or $\S^3$ action.  If $H^i(X;R)=0$ for all $i>n$ then $H^i(X/G;R)=0$ for all $i> n$.
\end{proposition}

\begin{proposition}(\cite{Hatcher})
	Let $R$ denotes a ring and $\S^{n-1}\to E\stackrel{p}{\rightarrow} B$ be an oriented sphere bundle. Then the following sequence is exact  with  coefficients in $R$
	\begin{align*}
	\cdots\rightarrow H^{i}(E)\stackrel{\rho}{\rightarrow}H^{i-n+1}(B)\stackrel{\cup}{\rightarrow}H^{i+1}(B) \stackrel{p^*}{\rightarrow}H^{i+1}(E)\stackrel{\rho}{\rightarrow}H^{i-n+2}(B)\rightarrow\cdots
	\end{align*}
	which  start with 
	\begin{align*}
	0\rightarrow &H^{n-1}(B)\stackrel{p*}{\rightarrow}H^{n-1}(E)\stackrel{\rho}{\rightarrow}H^0(B)\stackrel{\cup}{\rightarrow}H^n(B)\stackrel{p^*}{\rightarrow}H^n(E)\rightarrow\cdots 	\end{align*}
	where $\cup:H^i(B)\to H^{i+n}(B)$  maps $x\to x\cup u$ and $u\in H^n(B)$ denotes the characteristic class of the  sphere  bundle. This sequence is called the Gysin sequence. Observe that $p^*:H^i(E)\to H^i(B)$ is an isomorphism for all $0\leq i< n-1$.
\end{proposition}

\section{Index and co-index of $\mathbb{S}^1$ and $\mathbb{S}^3$-spaces}
Recall that for $G=\mathbb{S}^d, d=1$ or 3, the universal principal $\mathbb{S}^d$-bundle is $\mathbb{S}^d\hookrightarrow \mathbb{S}^{\infty} \to \mathbb{FP}^{\infty}$, where $\mathbb{F}=\C$ or $\mathbb{H}$. Let $G=\mathbb{S}^d, d=1$ or 3 acts freely on $\mathbb{S}^{(d+1)k+d}$ by the standard action namely the component wise multiplication. Note that for a  topological group $G$ and $G$-space $X$  there exist a $G$-equivariant  map $f:G\to X$ defined as $f(g)=g.x_0$ where $x_0$ is any fixed element of $X$ and $G$  acts on itself by the group multiplication. In particular, for a $G$-space $X$ there exist  an $G$-equivariant map $f:G\to X$.

Now, we define index  and co-index of a free $G$-space $X$. We denote it by $\text{ind}_G X$ and $\text{co-ind}_G X$, respectively.

\begin{definition}
	Let $X$ be a free $G$-space, where $G=\mathbb{S}^d$, $d=1$ or $3$.  We define index  of $X$ as
	\begin{align*}
	\text{ind}_G X=\max \{k| \text{ there exists an $G$-equivariant map } f:\mathbb{S}^{(d+1)k+d}\to X,\ k\geq 0\}.
	\end{align*}
It is easy to see that $\text{ind}_G X\geq 0$. If there is no upper bound for  equivariant maps from spheres to $X$ then $\text{ind}_G X=+\infty$.   
\end{definition}

\begin{definition}
			Let $X$ be a free $G$-space, where$G=\mathbb{S}^d$, $d=1$ or $3$. We define co-index of $X$ as
	\begin{align*}
	\text{co-ind}_G X=\min \{k| \text{ there exist  an $G$-equivariant map } f:X\to \mathbb{S}^{(d+1)k+d},\ k\geq 0\}.
	\end{align*}
	If no such $k$ exist then $\text{co-ind}_G X=+\infty$.   
\end{definition}

We define $\text{ind}_G X=-1$ or $\text{co-ind}_G X=-1$ if and only if $X=\emptyset$ where $G=\mathbb{S}^1$ or $\mathbb{S}^3$.

Now, we recall\cite{JKaur2015,Hemant2012} the mod 2 cohomology index of a paracompact Hausdorff free $G$-space, where  $G=\mathbb{S}^1$  or $\mathbb{S}^3$. It is similar to the homology index for free involution on a paracompact Hausdorff space defined by  Conner and Floyd\cite{ConnerFloyd1960}.

\begin{definition}
			Let $X$ be a free $G$-space, where $G=\mathbb{S}^d$, $d=1$ or $3$ and $\omega\in H^{d+1}(X/G)$ be the Stiefel-Witney class of the associated  $G$-bundle $X\rightarrow X/G$. Then the mod 2 cohomology $G$-index of $X$ is defined as  the largest integer $k$ such that $\omega^k\neq 0.$ It is deonoted by cohom-index$_{G} X$.   
\end{definition}

Next, we observe that the  index and co-index of    $\mathbb{S}^{(d+1)n+d}$ for the action of componentwise multiplication of  $\S^d$   is $n$, where $d=1$ or 3.

Let $\mathbb{S}^{(d+1)n+d}$ be a free $\mathbb{S}^d$-space, $d=1$ or 3, with the standard action. For all $k\leq n$, the inclusion map $i:\mathbb{S}^{(d+1)k+d}\to \mathbb{S}^{(d+1)n+d}$ is $\mathbb{S}^d$-equivariant, therefore, $\text{ind}_{\mathbb{S}^d}\mathbb{S}^{(d+1)n+d}\geq k$.  By the  Borsuk-Ulam theorem, there does not exist an $\mathbb{S}^d$-equivariant map $f:\S^{(d+1)k+d}\to \S^{(d+1)n+d}$ for $k>n$. This implies that  $\text{ind}_{\mathbb{S}^d}\mathbb{S}^{(d+1)n+d}=n.$ By the similar argument, $\text{co-ind}_{\mathbb{S}^d}\mathbb{S}^{(d+1)n+d}=n.$

We know that if $X$ is a finitistic free  $G$-space, where $G=\S^d, d=1$ or $3$, with mod 2 cohomology ring of $\S^{(d+1)n+d}$ then the orbit space $X/G\sim_2 \mathbb{FP}^n$, where $\mathbb{F}=\C$ or  $\mathbb{H}$. This gives that the mod 2 cohomology index of $X$ which admits free $G$ actions is $n$(\cite{JKaur2015}).

Next, we observe  some properties of the index and co-index of a $G$-space $X$  for $G=\mathbb{S}^1$ or $\mathbb{S}^3$:
\begin{theorem}
		Let $X$ and $Y$ be  free $G$-spaces, where $G=\mathbb{S}^d$, $d=1$ or $3$.
	\begin{enumerate}
	\item If  $f:X\to Y$ is $G$-equivariant map then $\text{ind}_{G}X\leq \text{ind}_{G}Y.$
	
	\item If  $f:X\to Y$ is $G$-equivariant map then $\text{co-ind}_{G}X\leq \text{co-ind}_{G}Y.$
		
		\item  $\text{ind}_{G}X\leq \text{co-ind}_{G}X.$

		\item $\text{co-ind}_{G}(X*Y)\leq \text{co-ind}_{G}X+\text{co-ind}_{G}Y+1.$ 
	\end{enumerate}
\begin{proof}
	(1) and (2) follows from the definitions.\\
	\noindent (3) Let $f:\mathbb{S}^{(d+1)k+d} \to X$ be any $\mathbb{S}^d$-equivariant map where $d=1 $ or $3$. Then, we have
	$
	k=\text{co-ind}_{\mathbb{S}^d}\mathbb{S}^{(d+1)k+d}\leq \text{co-ind}_{\mathbb{S}^d }X.$ Consequently, $\text{ind}_{\mathbb{S}^d}X\leq \text{co-ind}_{\mathbb{S}^d}X.$
	
	\noindent (4) If either $\text{co-ind}_{\mathbb{S}^d}X$ or $\text{co-ind}_{\mathbb{S}^d}Y$ are $+\infty$ then the result is trivially true. So, let $\text{co-ind}_{\mathbb{S}^d}X=m$ and $\text{co-ind}_{\mathbb{S}^d}Y=n$. Let $f:X\to \mathbb{S}^{(d+1)m+d}$ and $g:Y\to \mathbb{S}^{(d+1)n+d}$ be the corresponding $\mathbb{S}^d$-equivariant maps. Then the map $f*g:X*Y\to \mathbb{S}^{(d+1)m+d}*\mathbb{S}^{(d+1)n+d}\approx \mathbb{S}^{(d+1)(n+m+1)+d}$ defined by $f*g([x,y,t])=[f(x),g(y),t]$ is a  $\mathbb{S}^d$-equivariant map. Hence, we get
$	
\text{co-ind}_{\mathbb{S}^d}X*Y\leq\text{co-ind}_{\mathbb{S}^d}{\mathbb{S}^{(d+1)(n+m+1)+d}}
=\text{co-ind}_{\mathbb{S}^d}X+\text{co-ind}_{\mathbb{S}^d}Y+1. 
$\qedhere
\end{proof}			\end{theorem}

	Let  $G=\mathbb{S}^d$, $d=1$ or $3$  acts freely on $\mathbb{F}$ under the action of scalar multiplication, where $\mathbb{F}=\C$ or $\mathbb{H}$.  Then, $.:\mathbb{F}\times G \to \mathbb{F}$ defined by $(w,c)\rightarrow c^{-1}.w$
is a free right $G$ action on $\mathbb{F}$. Let $\lambda=(\frac{X\times\mathbb{F}}{G},q,X/G,\mathbb{F})$ be the associated line bundle of the principal $G$-bundle  $G\hookrightarrow X \stackrel{p}{\rightarrow} X/G$, where $q:\frac{X\times\mathbb{F}}{G}\to X/G$ is defined as $q([x,z])=p(x)=[x]$.  As $E_G$ is free $G$-space, the map $h:X_G\rightarrow X/G$ defined by  $h([x,e])=[x]$ is a fibre bundle associated to the principal $G$-bundle $G\hookrightarrow X\to X/G$. As $E_G$ is contractible, $h$ is a homotopy equivalence. 
Define $E_1(\gamma)=\{([v],tv)\in \mathbb{FP}^{\infty}\times \mathbb{F}^{\infty} \mid t\in \mathbb{F}\}$ and $r:E_1(\gamma)\to \mathbb{FP}^{\infty}$ is the restriction of projection map onto the first component. Then, $\xi=(E_1(\gamma), r, \mathbb{FP}^{\infty}, \mathbb{F})$ is a universal vector bundle  of rank 1.

Let $E$ denotes the total space of the pull back bundle $h^*(\lambda)$. We can define $\phi:E\rightarrow E_1(\gamma)$ by
$\phi([x,v],[x',c'])=([v],c'.\theta(x,x').v)$
where $\theta$ is the translation map of the principal $G$-bundle $p: X\to X/G$. For the Borel fibration $X\stackrel{i}{\hookrightarrow}X_G\stackrel{\pi}{\rightarrow}\mathbb{F}P^\infty$, it is easy to see that $\pi^*(\xi)\cong h^*(\lambda)$. So, we get the following lemma:    

\begin{lemma}\label{classifying map}
		Let $X$ be a free $G$-space, where $G=\mathbb{S}^d$, $d=1$ or $3$. Then the map $\pi\circ h': X/G\rightarrow \mathbb{FP}^{\infty }$ is a classifying map for the bundle $\lambda$ where $h':X/G \rightarrow X_G$ denotes a homotopy inverse of $h$. 
\end{lemma}
For each $k\geq 1$, $\mathbb{FP}^k\subseteq \mathbb{FP}^{\infty}$. Put $\xi'=\xi|_{\mathbb{FP}^k}$, the restriction bundle of the universal bundle $\xi$ at $\mathbb{CP}^k$.
Note that $\xi'=\left( r^{-1}(\mathbb{FP}^k), r, \mathbb{FP}^k,\mathbb{F}\right) $ is a tautological bundle of rank 1, where $r^{-1}(\mathbb{FP}^k)=\left\lbrace  ([z],tz)\in \mathbb{FP}^{k}\times \mathbb{F}^{k+1}\mid  t\in \mathbb{F}  \right\rbrace.$ It is easy to prove that 
\begin{lemma}\label{restriction of bundle}
		Let $X$ be a free $G$-space, where $G=\mathbb{S}^d$, $d=1$ or $3$, and $f:\mathbb{S}^{(d+1)k+d}\rightarrow X$ be $G$-equivariant map, where $d=1$ or 3. Then $\overline{f}^*(\lambda)\cong \xi'$ where $\overline{f}:\mathbb{FP}^k\to X/G$ is a continuous map induced by $f$.
	\end{lemma}
Now, we prove a Borsuk-Ulam type result for a finitistic  space  which admits  free actions of $G=\S^1$ or $\S^3$.
\begin{theorem}
	Let $X$ be a finitistic free $G$-space, where $G=\mathbb{S}^d$, $d=1$ or $3$. If the mod 2 cohomology index of $X$ is $n$, then there does not exist  $G$-equivariant	map $f:\mathbb{S}^{(d+1)k+d}\to X$ for all $k>n$, where $\mathbb{S}^{(d+1)k+d} $ equipped with action of componentwise multiplication of $\S^d$.
\begin{proof} We  prove for $G=\S^3$. Suppose, on the contrary, that there exists a
$G$-equivariant map $f:\mathbb{S}^{4k+3}\to X$ for $k>n$. Then $f$ induces a continuous map
$\overline{f}:\mathbb{HP}^k\to
X/G$  and let $p:X\rightarrow X/G$ be the principal $G$-bundle. Note that the Borel fibration $X\stackrel{i}{\hookrightarrow} X_G\stackrel{\pi}{\rightarrow} \mathbb{HP}^{\infty} $ is a fibre bundle with  structure group $G$. Let $h:X_{G}\to X/G$ be homotopy equivalence and $\xi=(E_1(\gamma ),\gamma,\mathbb{HP}^{\infty},\mathbb{H}) $ be universal quaternion vector bundle of rank 1, where 
$E_1(\gamma)=\{([v],tv)\in \mathbb{HP}^{\infty}\times \mathbb{H}^{\infty}\mid t\in \mathbb{H}  \} $ 
and $\gamma: E_1(\gamma)\to \mathbb{HP}^{\infty}, \gamma([v],tv)=[v]$ be the projection map. Treat it as  a real vector bundle of rank 4. Let $\xi'$ be the restriction bundle of $\xi$ at $\mathbb{HP}^k.$ By Lemma \ref{classifying map}, $\pi\circ h': X/G\to \mathbb{HP}^{\infty}$ is a classifying map for associated line bundle  $\lambda=\left( \frac{X\times \mathbb{H}}{G},q,X/G, \mathbb{H}\right) $  of principal $G$-bundle $p:X\to X/G$, where $h'$ denotes the homotopy inverse of $h$.  
We know that $H^*(\mathbb{HP}^{\infty})=\mathbb{Z}_2[t]$, where $\deg t=4$. So the  first nonzero  Steifel-Witney class of $\xi$ is $\omega_4(\xi)=t$. Put $\pi^*(t)=u$. By the naturality of Witney classes, we get 
$(h')^*(u)=( h')^*(\pi^*(t))=(\pi\circ h')^*(\omega_4(\xi))=\omega_4(\lambda).$
As $f$ is equivariant map, by Lemma \ref{restriction of bundle}, we get $\xi'\cong \overline{f}^*(\lambda)$ as real vector bundles. Again by the naturality of Witney classes, 
$$\overline{f}^*(\omega_4(\lambda))= \omega_4(\overline{f}^*(\lambda))= \omega_4(\xi')=t',\text{ ~~~~~~~~~~~~~~~~~~~~~~~~say. }$$
So, $(h')^*(u)=\omega_4(\lambda)$ and $\overline{f}^*(\omega_4(\lambda))=t'\neq 0$. So, we get  $t'=\overline{f}^*((h')^*(u))$ and hence $t'$  is the generator of $ H^4(\mathbb{HP}^k)$. Note that $(h')^*(u)$ is the characteristic class for the principal bundle $X\to X/G$. By our hypothesis, the mod 2 cohomology index of $X$ is $n$ which gives that $((h')^*(u))^{n+1}=0$. So, we get  $0=\overline{f}^*(((h')^*(u))^{n+1})= t'^{n+1}$ which contradicts that $t'$ is the  generator of $H^4(\mathbb{HP}^{k})$.

Similarly, we can prove for $G=\S^1$.
\end{proof}
\end{theorem}
From the above theorem, it is easy to derive
\begin{corollary}\label{ind<coind}
	 	Let $X$ be a finitistic free $G$-space, where $G=\mathbb{S}^d$, $d=1$ or $3$. Then 
	$$\text{ind}_{G}X\leq \text{cohom-ind}_{G}X.$$
\end{corollary}

\section{$\mathbb{S}^1$ and $\mathbb{S}^3$ actions on product of spheres}
In this section, we will compute  the cohomology structure  of the orbit spaces of free actions of $G=\mathbb{S}^1$ or $\S^3$  on a finitistic space $X\sim_R \S^n\times\S^m,1\leq n\leq m$, where $R=\mathbb{Q}$ or $\mathbb{Z}_2$. We also obtain an upper bound of the index of $X\sim_{\mathbb{Z}_2} \S^n\times\S^m$ for free $G$ actions and establish Borsuk-Ulam type results. By the Kunneth formula, we have
$H^*(X)=R[x,y]/\langle x^2,y^2 \rangle,\quad \deg x=n,\,\deg y=m.  
$

First, we discuss free actions of $G=\mathbb{S}^3$ on $X$ with the rational coefficients.

\begin{theorem}
	Let $G=\mathbb{S}^3$ acts freely on a finitistic space $X\sim_{\mathbb{Q}}\mathbb{S}^n\times \mathbb{S}^m,1\leq n\leq m$.  Then $H^*(X/G)$ is isomorphic to one of the following:
	\begin{enumerate}
		\item[(i)] $\mathbb{Q}[u,v]\langle u^{\frac{n+1}{4}},v^2 \rangle$,
		where $\deg u=4,\deg v=m,n\equiv3(\text{mod } 4)$

		\item[(ii)] $\mathbb{Q}[u,v]/\langle u^{\frac{m+n+1}{4}},u^{\frac{m-n+1}{4}}v-\alpha u^{\frac{m+1}{4}},v^2-\beta u^{\frac{n}{2}}-\gamma u^{\frac{n}{4}}v\rangle$,
		where $\deg u=4,\deg v=n,\alpha,\beta,\gamma\in \mathbb{Q}$, $m-n\equiv 3(\text{mod } 4)$, $n$ is even, $m$ is odd and $\alpha=0$ if $m<2n$ or $n\equiv 2(\text{mod } 4)$ and $\gamma=0$ if $n\not\equiv 0( \text{mod } 4)$
		
		\item[(iii)] $\mathbb{Q}[u,v]/\langle u^{\frac{m+1}{4}},v^2 -\alpha u^{\frac{n}{2}}-\beta u^{\frac{n}{4}}v \rangle$, where $\deg u=4,\deg v=n,m\equiv3(\text{mod } 4),\alpha,\beta\in\mathbb{Q},\beta=0$ if $n\not\equiv0(\text{mod } 4)$ and $\alpha=0$ if either $n\not\equiv2(\text{mod } 4)$ or $2n\geq m$.
	\end{enumerate}	
\end{theorem}
\begin{proof}
	As  $\pi_1(B_G)$ acts trivially on $X$,
		$E_{2}^{k,l}\cong H^k(B_G)\otimes H^l(X)$. First, we assume  $n<m$.	
	The possible nontrivial differentials are $d_{n+1}, d_{m+1},d_{m-n+1}$ or $d_{m+n+1}$.   There are two possible cases: (i) $d_{r}(1\otimes x)\neq 0$, and  (ii) $d_{r}(1\otimes x)=0$ and $d_{r}(1\otimes y)\neq 0$.

	\textbf{Case (i):} When $d_{r}(1\otimes x)\neq 0$, then  $r$ must be $n+1$, $n\equiv3(\text{mod } 4) $ and $d_{n+1}(1\otimes x)=ct^{\frac{n+1}{4}}\otimes 1$ for some $0\neq c\in\mathbb{Q} $. If $m=2n$ and $d_{n+1}(1\otimes y)=dt^{\frac{n+1}{4}}\otimes x$, where $0\neq d\in\mathbb{Q} $, then
	$0=d_{n+1}(1\otimes y^2)=2dt^{\frac{n+1}{4}}\otimes xy$ which is not possible. Therefore, $d_{n+1}(1\otimes y)=0$. Consequently, $d_{n+1}(t^k\otimes x)=ct^{k+\frac{n+1}{4}}\otimes 1$, $d_{n+1}(t^k\otimes y)=0$ and $d_{n+1}(t^k\otimes xy)=ct^{k+\frac{n+1}{4}}\otimes y$   for all $k\geq 0$, and  $E_2^{*,*}=E_{n+1}^{*,*}$.  This implies that
	$d_{n+1}:E_{n+1}^{k,l}\to E_{n+1}^{k+n+1,l-n}$  is an isomorphism for all $k\geq 0$, and $l=n$ or $n+m$. So, we get $E_{n+2}^{k,l}=0$ for all $k\geq 0$, $l=n$ or $n+m$; and  $E_{n+2}^{k,l}=E_{2}^{k,l}$ if $k<n+1$, $l=0$ or $m$ and trivial otherwise. As $n<m$, $d_{m+1}$ is the trivial homomorphism and so  $E^{*,*}_{\infty}=E_{n+2}^{*,*}$. So, we have
		\begin{align*}
	H^j(X_G)=
	\begin{cases}
	\mathbb{Q}&\text{ if }  0\leq j\equiv0(\text{mod } 4)< n,m\leq j\equiv m(\text{mod }4)<m+n\\
	0&\text{ otherwise.}
	\end{cases} 
	\end{align*}   Note that $t\otimes 1\in E_{2}^{4,0}$ and $1\otimes y\in E_{2}^{0,m}$ are permanent cocyles, where $t$ generates $H^*(B_G)$. 
 Let $w\in E_{\infty}^{0,m}$ and  $u\in E_{\infty}^{4,0}$ be elements corresponding to $1\otimes y\in E_{2}^{0,m}$ and $t\otimes 1\in E_{2}^{4,0}$, respectively. Then $w^2=0$ and $ u^{\frac{n+1}{4}}=0$.  
	Therefore, there exist $v\in H^m(X_G)$ such that $i^*(v)=y$. So, we have $v^2=0$. Thus, 
	$$H^*(X_G)=\frac{\mathbb{Q}[u,v]}{\langle u^{\frac{n+1}{4}},v^2 \rangle}$$
	where $\deg u=4,\deg v=m.$ This realizes case(i).

	\noindent\textbf{Case(ii):} When $d_r(1\otimes x)=0$ and $d_r(1\otimes y)\neq 0$. Then, there  are two subcases:

	\textbf{Subcase(i):} When $r=m-n+1$.
	
	In this case, $m-n\equiv3(\text{mod } 4) $ and $d_{m-n+1}(1\otimes y)=ct^{\frac{m-n+1}{4}}\otimes x$ for some $0\neq c\in\mathbb{Q} $. If $m$ is even then  $0=d_{m-n+1}(1\otimes y^2)=2ct^{\frac{m-n+1}{4}}\otimes xy$ which is not possible. Therefore, $m$ must be odd and $n$ is even. 
	We get $d_{m-n+1}(t^k\otimes xy)=0$, $d_{m-n+1}(t^k\otimes y)=ct^{k+\frac{m-n+1}{4}}\otimes x$   for all $k\geq 0$, and   $E_2^{*,*}=E_{m-n+1}^{*,*}$.  This implies that
	$d_{m-n+1}:E_{m-n+1}^{k,m}\to E_{m-n+1}^{k+m-n+1,n}$  is an isomorphism for all $k\geq 0$. So, $
	E_{m-n+2}^{k,l}= E_{2}^{k,l}\text{ for all } k\geq 0,l=0,$ $n+m;  
	E_{m-n+2}^{k,n}= E_{2}^{k,n} $  if $k<m-n+1$ and trivial otherwise.
	As $G$ acts freely on $X$,  $d_{m+n+1}$ must be nontrivial. Let $d_{m+n+1}(1\otimes xy)=dt^{\frac{m+n+1}{4}}\otimes 1$, where $0\neq d\in \mathbb{Q}$. Then, $d_{m+n+1}(t^k\otimes xy)=dt^{k+\frac{m+n+1}{4}}\otimes 1$ for all $k\geq 0$, and so $d_{m+n+1}:E_{m+n+1}^{k,m+n}\to E_{m+n+1}^{k+m+n+1,0}$  is an isomorphism for all $k\geq 0$.  This implies that $ E_{m+n+2}^{k,n}=E_{m-n+2}^{k,n},
	E_{m+n+2}^{k,m+n}=0$ for all $k\geq 0$, and $E_{m+n+2}^{k,0}=E_{m-n+2}^{k,0}$ if $k<m+n+1$ and trivial otherwise.  Now, 	$E^{*,*}_{\infty}=E_{m+n+2}^{*,*}$.      If $n\equiv 0(\text{mod } 4)$ then $m\equiv3(\text{mod } 4) $, and we get  
	\begin{align*}
	H^j(X/G)=
	\begin{cases}
	\mathbb{Q}&\text{ if }0 \leq j\equiv0(\text{mod } 4)<n, m<j\equiv0(\text{mod } 4)<n+m\\
	\mathbb{Q}\oplus\mathbb{Q}&\text{ if }n\leq j\equiv0(\text{mod } 4)<m\\
	0&\text{ otherwise. } \end{cases} 
	\end{align*}
	If $n\equiv 2(\text{mod } 4)$ then $m\equiv1(\text{mod } 4) $, and we get 
	\begin{align*}
	H^j(X/G)=
	\begin{cases}
	\mathbb{Q}&\text{ if } 0 \leq j\equiv0(\text{mod } 4)<n+m,n\leq j\equiv2(\text{mod } 4)<m\\
	0&\text{ otherwise. }\end{cases} 
	\end{align*}  
	  
	 Let $w\in E_{\infty}^{0,n}$ and  $u\in E_{\infty}^{4,0}$ be elements corresponding to  permanent cocycles $1\otimes x\in E_{2}^{0,n}$ and $t\otimes 1\in E_{2}^{4,0}$, respectively. We have  $w^2=u^{\frac{m+n+1}{4}}=u^{\frac{m-n+1}{4}}w=0$. This implies that 
	$$\textnormal{Tot}E_{\infty}^{*,*}\cong \frac{\mathbb{Q}[u,w]}{\langle w^2,u^{\frac{m+n+1}{4}},u^{\frac{m-n+1}{4}}w \rangle}\text{, where} \deg u=4\text{ and } \deg w=n.$$ 
	There exist $v\in H^n(X_G)$ corresponding to $w\in E_{\infty}^{0,n}$ such that $i^*(v)=x$. 
	As $u^{\frac{m-n+1}{4}}w=0$ in $\textnormal{Tot}E_{\infty}^{*,*}$, we get  $u^{\frac{m-n+1}{4}}v=\alpha u^{\frac{m+1}{4}}$, $\alpha\in\mathbb{Q}$ and $\alpha=0$ when $n<2m$ or $n\equiv2(\text{mod } 4)$.  Also, if $2n\leq m+n-3$ then $u^{\frac{n}{2}}\in H^{2n}(X_G)$ and if $n$ is multiple of 4 then $u^{\frac{n}{4}}v\in H^{2n}(X_G)$. Therefore, $v^2=\alpha u^{\frac{n}{2}}+\beta u^{\frac{n}{4}}v$ where $\alpha,\beta\in \mathbb{Q}$. Thus, we have
	\begin{align*}
	H^*(X_G)=\frac{\mathbb{Q}[u,v]}{\langle u^{\frac{m+n+1}{4}},u^{\frac{m-n+1}{4}}v-\alpha u^{\frac{m+1}{4}},v^2-\beta u^{\frac{n}{2}}-\gamma u^{\frac{n}{4}}v\rangle}
	\end{align*}
	where $\deg u=4,\deg v=n,\alpha,\beta,\gamma\in \mathbb{Q}$, $m-n\equiv 3(\text{mod } 4)$, $n$ is even, $m$ is odd and $\alpha=0$ if $m<2n$ or $n\equiv2(\text{mod } 4)$ and $\gamma=0$ if $n\not\equiv0(\text{mod } 4)$.  This realizes case(ii) of the theorem.

	\textbf{Subcase (ii):} When $r=m+1$.

	In this case,  $m\equiv3(\text{mod } 4) $ and $d_{n+1}=d_{m-n+1}\equiv 0$. Let  $d_{m+1}(1\otimes y)=ct^{\frac{m+1}{4}}\otimes 1$ for some $0\neq c\in\mathbb{Q} $.  Then $d_{m+1}(1\otimes x)=0$ and $d_{m+1}(t^k\otimes xy)=(-1)^nct^{k+\frac{m+1}{4}}\otimes y$   for all $k\geq 0$ and  $E_2^{*,*}=E_{m+1}^{*,*}$.  This implies that
	$d_{m+1}:E_{m+1}^{k,l}\to E_{m+1}^{k+m+1,l-m}$  is an isomorphism for all $k\geq 0$, and $l=m$ or $m+n$. So, we get  $E_{m+2}^{k,l}=E_{2}^{k,l}$ if $k<m+1$ and $l=0$ or $n$ and trivial otherwise.  So, $E_{\infty}^{*,*}=E_{m+2}^{*,*}$.  If $n\equiv0(\text{mod } 4)$, then
	\begin{align*}
	H^j(X_G)=
	\begin{cases}
	\mathbb{Q}&\text{ if }  0\leq j\equiv0(\text{mod } 4)< n,m<j\equiv0(\text{mod } 4)<m+n\\
	\mathbb{Q}\oplus\mathbb{Q}&\text{ if }n\leq j\equiv0(\text{mod } 4)<m\\
	0&\text{ otherwise.}
	\end{cases} 
	\end{align*} 
	If $n\not\equiv0(\text{mod } 4)$, then
	\begin{align*}
	H^j(X_G)=
	\begin{cases}
	\mathbb{Q}&\text{ if }  0\leq j\equiv0(\text{mod } 4)< m, n\leq j\equiv n(\text{mod } 4)<n+m\\
	0&\text{ otherwise.}
	\end{cases} 
	\end{align*} 
	 Let $w\in E_{\infty}^{0,n}$ and  $u\in E_{\infty}^{4,0}$ be elements corresponding to permanent cocycles $1\otimes x\in E_{2}^{0,n}$ and $t\otimes 1\in E_{2}^{4,0}$, respectively. Then $w^2=0$ and $ u^{\frac{m+1}{4}}=0$. There exist $v\in H^n(X_G)$ corresponding to $w\in E_{\infty}^{0,n}$ such that $i^*(v)=x$. We have  $v^2 =\alpha u^{\frac{n}{2}}+\beta u^{\frac{n}{4}}v$, where $\alpha,\beta\in\mathbb{Q}$. Thus, 
	$$H^*(X_G)=\frac{\mathbb{Q}[u,v]}{\langle u^{\frac{m+1}{4}},v^2 -\alpha u^{\frac{n}{2}}-\beta u^{\frac{n}{4}}v \rangle}$$
	where $\deg u=4,\deg v=n,\alpha,\beta\in\mathbb{Q},\beta=0$ if $n\not\equiv0(\text{mod } 4)$ and $\alpha=0$ if either $n\not\equiv2(\text{mod } 4)$ or $2n\geq m$. This realizes case(iii) of the theorem.

	Now, we assume $n=m$. If $d_{n+1}(1\otimes x)=ct^{\frac{n+1}{4}}\otimes1$ and $d_{n+1}(1\otimes y)=0$ for some $0\neq c\in\mathbb{Q}$, then it is same as case(i). Now, suppose $d_{n+1}(1\otimes x)=ct^{\frac{n+1}{4}}\otimes1$ and $d_{n+1}(1\otimes y)=dt^{\frac{n+1}{4}}\otimes 1$ for some $0\neq c,d\in\mathbb{Q}$. Then, $n\equiv 3(\text{mod } 4)$,  $d_{n+1}(1\otimes (c_1 x+c_2 y))=t^{\frac{n+1}{4}}\otimes (c_1c+c_2 d)$, $d_{n+1}(1\otimes xy)=t^{{\frac{n+1}{4}}}\otimes(cy-dx)$ and   $E_2^{*,*}=E_{n+1}^{*,*}$.  This implies that $d_{n+1}:E_{n+1}^{k,2n}\to E_{n+1}^{k+n+1,n}$ is injective; and
	$\ker(d_{n+1}:E_{n+1}^{4k,n}\to E_{n+1}^{4k+n+1,0})\cong\mathbb{Q}$ and $\im (d_{n+1}:E_{n+1}^{4k-n-1,2n}\to E_{n+1}^{4k,n})\cong \mathbb{Q}$  with basis $\{t^k\otimes(cy-dx)\}$ and $\{t^{k-\frac{n+1}{4}}\otimes (cy-dx)\}$, respectively, for all $k\geq 0$. So, we get $E_{n+2}^{4k,0}=E_{n+2}^{4k,n}\cong \mathbb{Q}$ with basis $\{t^k\otimes 1\}$ and $\{t^k\otimes (cy-dx)\}$, respectively, for all $0\leq k<\frac{n+1}{4}$; and trivial otherwise.  So, $E_{\infty}^{*,*}=E_{n+2}^{*,*}$. 
	So, we have
	\begin{align*}
	H^j(X_G)=
	\begin{cases}
	\mathbb{Q}&\text{ if }  0\leq j\equiv0(\text{mod } 4)< n, n\leq j\equiv n(\text{mod } 4)<2n\\
	0&\text{ otherwise.}
	\end{cases} 
	\end{align*} 	
	 Let $w\in E_{\infty}^{0,n}$ and  $u\in E_{\infty}^{4,0}$ be elements corresponding to permanent cocycles $1\otimes (cy-dx)\in E_{2}^{0,n}$ and $t\otimes 1\in E_{2}^{4,0}$, respectively. Then  $w^2=0$ and $u^{\frac{n+1}{4}}=0$. There exist $v\in H^n(X_G)$ corresponding to $w\in E_{\infty}^{0,n}$ such that $i^*(v)=cy-dx$. We have $v^2=0$, and thus
	 $$H^*(X_G)=\frac{\mathbb{Q}[u,v]}{\langle u^{\frac{n+1}{4}},v^2 \rangle}$$
	where $\deg u=4,\deg v=n.$  This realizes case(i) of the theorem.
	\end{proof}

Now, we discuss free actions of $G=\mathbb{S}^3$
on $X\sim_2\mathbb{S}^n\times\mathbb{S}^m$. First, we derive following results.
\begin{lemma}
	Let $X$ be a  finitistic free $G$-space with $X\sim_2\S^n\times \S^m,1\leq n\leq m$, where $G=\S^d$, $d=1$ or $3$. Then $H^i(X/G)=0$ for all $i\geq n+m-(d-1)$.\label{L1S13}
	\begin{proof}
	By Proposition \ref{S1 and S3 pre}, $H^i(X/G)=0$ for all $i> n+m$. For $n+m-(d-1)\leq i\leq n+m$, the result follows 	by taking $n+m+1\leq k\leq n+m+d$ in the Gysin sequence of the sphere bundle $G\hookrightarrow X\to X/G$.
		\end{proof}	
\end{lemma}

\begin{lemma}
	Let $X$ be a  finitistic free $G$-space with $X\sim_2\S^n\times \S^m,1\leq n\leq m$, and $p:X\to X/G$ be the orbit map, where $G=\S^d$, $d=1$ or $3$. Then $p^*:H^i(X/G)\to H^i(X)$ cannot be nontrivial for both $i=n$ and $m$. \label{L2S3}
	\begin{proof}
		Assume otherwise. Then, there exist nonzero elements $u\in H^n(X/G)$ and $v\in H^m(X/G)$ such that $p^*(u)=x$ and $p^*(v)=y$. Consequently, $p^*(uv)=xy\neq 0$ in $H^{n+m}(X)$. This implies that $uv\neq 0$ in $H^{n+m}(X/G)$, a contradiction. 
	\end{proof}
\end{lemma}
Now, we determine the orbit spaces of a finitistic space $X\sim_2 \S^n\times\S^m,1\leq n\leq m$ for $G=\S^d, d=1$ or 3, actions.
\begin{theorem}\label{TS3}
		Let $X$ be a  finitistic free $G$-space with $X\sim_2\S^n\times \S^m,1\leq n\leq m$, where $G=\S^3$. Then $H^*(X/G)$ is isomorphic to one of the following graded algebras:
	\begin{enumerate}
		\item[(i)] $\mathbb{Z}_2[u, v ]/\langle u^{\frac{m+1}{4}}, v^2+\alpha vu^{\frac{n}{4}}+\beta u^\frac{n}{2}\rangle$, where $\deg u=4$, $\deg v=n$, $m\equiv$3(mod 4), $\beta=0$ if  $ m<2n$ and $n$ is even; and $\alpha=\beta=0$ if $n$ is odd.

		\item[(ii)] $\mathbb{Z}_2[u,v]/\langle u^{\frac{n+m+1}{4}}, vu^{\frac{m-n+1}{4}}, v^2+\alpha vu^{\frac{n}{4}} +\beta u^{\frac{n}{2}} \rangle$, where $\deg u=4$, $\deg v=n$, $m-n\equiv 3$(mod 4), $m\equiv$3(mod 4) and $\alpha=0 $ if $m<2n$.
		
		\item[(iii)] $\mathbb{Z}_2[u, v]/\langle u^{\frac{n+1}{4}}, v^2\rangle$, where $\deg u=2$, $\deg v=m$ and $n\equiv 3$(mod 4).
	\end{enumerate}
	\begin{proof}
		Recall that the Gysin sequence of the  sphere bundle $G\hookrightarrow X\stackrel{p}{\longrightarrow} X/G$ is:
		\begin{align*}
		\cdots\longrightarrow H^{i}(X)\stackrel{\rho}{\longrightarrow}H^{i-3}(X/G)\stackrel{\cup}{\longrightarrow}H^{i+1}(X/G) \stackrel{p^*}{\longrightarrow}H^{i+1}(X)\stackrel{\rho}{\longrightarrow}H^{i-2}(X/G)\longrightarrow\cdots
		\end{align*}
		which begins with
		\begin{align*}
		0\longrightarrow &H^3(X/G)\stackrel{p*}{\longrightarrow}H^3(X)\stackrel{\rho}{\longrightarrow}H^0(X/G)\stackrel{\cup}{\longrightarrow}H^4(X/G)\stackrel{p^*}{\longrightarrow}H^4(X)\longrightarrow\cdots
		\end{align*}
Then, for $0<i<n-1$, $n<i<m-1$  and $m<i<n+m-1$, we have $H^{i-3}(X/G)\cong H^{i+1}(X/G)$.  	We also have  	$H^i(X/G)\cong H^i(X)$ for $0\leq i\leq 2$. This gives that $\Z_2\cong H^i(X/G)\cong H^{i+4}(X/G)$ for $0\leq i\equiv 0$(mod 4)$<n-4$ and $H^i(X/G)=0$ for $0<i\equiv j$(mod 4)$<n$, where $1\leq j\leq 3$. 
Let $u\in H^4(X/G)$ be the image of the generator $1\in H^0(X/G)$ under the homomorphism $\cup:H^0(X/G)\to H^4(X/G)$. So, inductively generator of $H^{i}(X/G)$ is $\{u^{\frac{i}{4}}\}$ for  all $0\leq i\equiv 0$(mod 4)$<n$.  As $\mathbb{S}^3$ acts freely on $X$, both $n$ and $m$  cannot be even \cite{Dotzel}. It is clear that isomorphisms appears in differences of mod 4, so we consider the following cases:

		\noindent\textbf{Case(I)}: When $n\equiv 1$(mod 4).

	 As $n\equiv 1$(mod 4), $H^{n-1}(X/G)\cong \mathbb{Z}_2$  with the basis 
		$\{u^{\frac{n-1}{4}}\}$ and $H^{n-j}(X/G)=0$ for  $2\leq j\leq 4$. This implies that $H^{n+1}(X/G)=H^{n+2}(X/G)=0$,  $H^n(X/G)\cong H^{n}(X)$ and $H^{n+3}(X/G)\cong H^{n-1}(X/G)$. So the basis for $H^n(X/G)$ and $H^{n+3}(X/G)$ are $\{v\}$ and $\{u^{\frac{n+3}{4}}\}$, respectively where $p^*(v)=x$. Consequently,  $H^{i}(X/G)\cong\Z_2 $ for $n\leq i\equiv j$(mod 4)$<m$, $j=0$ or 1; and $H^{i}(X/G)=0$ for  $n<i\equiv j$(mod 4)$<m$, $j=2$  or 3. Inductively, the bases  for $H^{i}(X/G)$ are  $\{vu^{\frac{i-n}{4}}\}$ when $n\leq  i\equiv 1$(mod 4)$<m$,   and $\{u^{\frac{i}{4}}\}$ when $n\leq i\equiv$0(mod 4)$<m$. Now, the possible  values of $m$ are as (a) $m\equiv$3(mod 4), (b) $m\equiv$2(mod 4), (c) $m\equiv$1(mod 4) and (d) $m\equiv$0(mod 4).

		\noindent \textbf{Subcase(a)}:  	
		As $m\equiv$3(mod 4), we get  $0=H^{m-1}(X/G)=H^{m-4}(X/G)$, and $\Z_2\cong H^{m-2}(X/G)\cong H^{m-3}(X/G)$. 
		By Lemma \ref{L2S3}, $p^*:H^m(X/G)\to H^m(X)$ must be trivial. 
		This implies that $H^{m+2}(X/G)\cong H^{m-2}(X/G)$, $H^{m+3}(X/G)=H^m(X/G)=0$ and $\rho:H^{m}(X)\to H^{m-3}(X/G)$ is an isomorphism. Consequently, $H^{m+1}(X/G)=0$.  We have  $H^{i-3}(X/G)\cong H^{i+1}(X/G)$ for $m<i<n+m-1$. So, we get $H^{m+n-j}(X/G)=0$ for $j=1,2$ and 4; and $H^{m+n-3}(X/G)\cong\Z_2$.	Accordingly, we get  
		\begin{align*}
		H^i(X/G)=\begin{cases}
		\mathbb{Z}_2 &\text{ if } 0\leq i\equiv 0\text{(mod 4)}<m,\,n\leq i\equiv 1\text{(mod 4)}<n+m\\
		0 &\text{ otherwise}
		\end{cases}
		\end{align*}
		If $m=n+2$ then we get similar cohomology groups. Note that the  basis of $H^{m-3}(X/G)$ is $\{u^{\frac{m-3}{4}}\}$. As the homomorphism $\cup:H^{m-3}(X/G)\to H^{m+1}(X/G)$ is trivial, we get $u^{\frac{m+1}{4}}=0$. Since $n\equiv 1$(mod 4), we get $H^{2n}(X/G)=0$. Consequently, $v^2=0$. Therefore, $H^*(X/G)$ is given by
		${\mathbb{Z}_2[u,v]}/{\langle u^\frac{m+1}{4},v^2\rangle} $
		where $\deg u=4,\deg v=n$.
		This realizes possibility (i).

		\noindent\textbf{Subcase(b):}		
		As $m\equiv$2(mod 4), we get  $H^{m-4}(X/G)=H^{m-3}(X/G)=0$ and $H^{m-1}(X/G)\\\cong H^{m-2}(X/G) \cong\mathbb{Z}_2$.
	By Lemma \ref{L2S3}, $p^*:H^m(X/G)\to H^m(X)$ must be trivial. 
		This implies that $H^{m}(X/G)=H^{m+1}(X/G)=0$, $H^{m+2}(X/G)\cong H^{m-2}(X/G)$ and $H^{m+3}(X/G)\cong H^{m-1}(X/G)$. Consequently, $H^{m+n-2}(X/G)\cong \mathbb{Z}_2$, which contradicts that $G$ acts freely on $X$. 	If $m=n+1$ then  we get same cohomology groups.

		\noindent\textbf{Subcase(c):} When $m\equiv 1$(mod 4). If $n\neq m$ then we get  $H^{m-3}(X/G)=H^{m-2}(X/G)$$\\=0$ and $H^{m-4}(X/G)\cong H^{m-1}(X/G)\cong\Z_2$.
		By the exactness of the Gysin sequence,  $p^*:H^{m}(X/G)\to H^m(X)$  must be nontrivial, a contradiction. If $n=m$ then  $H^n(X/G)\cong H^n(X)\cong\mathbb{Z}_2\oplus\mathbb{Z}_2$ and $H^{n+j}(X/G)=0$ for $1\leq j\leq 3$. This gives $H^{2n-1}(X/G)\cong \mathbb{Z}_2\oplus\mathbb{Z}_2$, a contradiction.
		
		\noindent\textbf{Subcase(d)}: 	
	As $m\equiv$0(mod 4), we get $H^{m-1}(X/G)=H^{m-2}(X/G)=0$ and $H^{m-4}(X/G)\\\cong H^{m-3}(X/G)\cong\Z_2$. 
		By Lemma \ref{L2S3}, $p^*:H^m(X/G)\to H^m(X)$ must be trivial.  
		This implies that  $H^m(X/G)\cong\mathbb{Z}_2$. Consequently, $H^{m+n-1}(X/G)\cong\Z_2$ which contradicts that $G$ acts freely on $X$. If $m=n+3$ then  we get the same cohomology groups.

		\noindent\textbf{Case(II)}: When $n\equiv$3(mod 4). First, we consider $n\neq m$.

		As $n\equiv$3(mod 4), we get $H^{n-3}(X/G)\cong\mathbb{Z}_2$  with the basis 
		$\{u^{\frac{n-3}{4}}\}$ and $H^{n-j}(X/G)$$\\=0$ for  $j=1,2,4$.  By the exactness of the Gysin sequence, $H^{n+j}(X/G)=0$ for $j=2$ and 3. There are two cases:

		\noindent\textbf{Subcase(i):} When $p^*:H^n(X/G)\to H^n(X)$ is trivial.

		In this case, $\rho:H^n(X)\to H^{n-3}(X/G)$ is an isomorphism and $H^{n+j}(X/G)=0$ for $j=0,1$.  Accordingly,  $H^i(X/G)=0$ for  $n\leq i<m$. 
		By the exactness of the Gysin sequence, $H^m(X/G)\cong H^m(X)$ and $H^{m+j}(X/G)=0$ for  $1\leq j\leq 3$. Consequently, $H^{m+n-3}(X/G)\cong\Z_2$ and  $H^{m+n-j}(X/G)=0$ for  $j=1,2$ and 4.
 Thus, we have  
		\begin{align*}
		H^i(X/G)=\begin{cases}
		\mathbb{Z}_2 &\text{ if } 0\leq i\equiv 0\text{(mod 4)}<n,\,m\leq i\equiv m\text{(mod 4)}<n+m\\
		0 &\text{ otherwise}
		\end{cases}
		\end{align*}
	As $\cup:H^{n-3}(X/G)\to H^{n+1}(X/G)$ is trivial, we get $u^{\frac{n+1}{4}}$=0.	Since $p^*:H^m(X/G)\to H^m(X)$ is an isomorphism, we have $p^*(v)=y$ where $\{v\}$ is the basis of $H^m(X/G)$. Now $n<m$ implies that  $v^2=0$.   Therefore, $H^*(X/G)$ is given by
		$
		{\mathbb{Z}_2[u,v]}/{\langle u^\frac{n+1}{4},v^2\rangle}$
		where $\deg u=4,\deg v=m$. This realizes possibility (iii).

		\noindent\textbf{Subcase(ii):} When $p^*:H^n(X/G)\to H^n(X)$ is nontrivial.

		Then $H^n(X/G)\cong H^{n}(X)$ and $H^{n+1}(X/G)\cong H^{n-3}(X/G)$. From this we get,  
		$H^i(X/G)\\=0$ for $n<i\equiv j$(mod 4)$<m$, $j=1$ or 2 and $H^i(X/G)\cong\Z_2$ for $n\leq i\equiv j$(mod 4)$<m$, $j= 0$ or 3. Inductively, the  bases for $H^i(X/G)$ are $\{u^{\frac{i}{4}}\}$ when $n<i\equiv 0$(mod 4)$<m$, and  $\{vu^{\frac{i-n}{4}}\}$ when $n\leq i\equiv 3$(mod 4)$<m$ where $v\in H^n(X/G)$ such that $p^*(v)=x$.
		There are four possibilities for $m$: (a) $m\equiv 3$(mod 4), (b) $m\equiv 2$(mod 4), (c) $m\equiv 1$(mod 4), and (d) $m\equiv 0$(mod 4)

		\noindent \textbf{Subcase(a):} As  $m\equiv 3$(mod 4), we get  
		 $ H^{m-2}(X/G)= H^{m-1}(X/G)=0$ and $ H^{m-3}(X/G)\\ \cong  H^{m-4}(X/G)\cong\Z_2$.
		By Lemma \ref{L2S3}, $p^*:H^m(X/G)\to H^m(X)$ must be trivial. This gives that $H^m(X/G)\cong H^{m-4}(X/G)$  and $H^{m+j}(X/G)=0$ for $1\leq j\leq 3$.  Consequently, we have $ H^{m+n-3}(X/G)\cong\Z_2$ and  $H^{m+n-j}(X/G)=0$ for $j=1,2$ and 4.	Thus, we have
		\begin{align*}
		H^i(X/G)=\begin{cases}
		\mathbb{Z}_2 &\text{ if } 0\leq i\equiv 0\text{(mod 4)}<m,\,n\leq i\equiv 3\text{(mod 4)}<n+m, \\
		0 &\text{ otherwise}
		\end{cases}
		\end{align*}
		As $\cup:H^{m-3}(X/G)\to H^{m+1}(X/G)$ is trivial, we get $u^{\frac{m+1}{4}}=0$. Since $n\equiv 3\text{(mod 4)}$,  $v^2=0$. Therefore, $H^*(X/G)$ is given by
		${\mathbb{Z}_2[u,v]}/{\langle u^\frac{m+1}{4},v^2\rangle} $
		where $\deg u=4,\deg v=n$.
		This realizes possibility (i).

\noindent \textbf{Subcase(b):} As $m\equiv$2(mod 4), we have $H^{m-2}(X/G)=0$. This implies that $H^{m+2}(X/G)\\ \cong H^{m-2}(X/G)$.  Consequently, $H^{m+n-1}(X/G)\cong\Z_2$ which contradicts that $G$ acts freely on $X$.	If $m=n+3$, then   we get the same cohomology groups.

\noindent \textbf{Subcase(c)\&(d):} For $m\equiv$1(mod 4) or $m\equiv$0(mod 4), we get $H^{m-3}(X/G)=0$. By the exactness of the Gysin sequence, $p^*:H^m(X/G)\to H^m(X)$ must be nontrivial,  a contradiction. 		We get the same result for $m=n+2$ or $m=n+1$.
		
	\noindent Now, we consider $n=m$.

	 It is clear that $H^{n+j}(X/G)=0, j=2,3$. As $H^n(X)\cong \mathbb{Z}_2\oplus\mathbb{Z}_2$, $\rho:H^n(X)\to H^{n-3}(X/G)$ cannot be injective and so $p^*:H^{n}(X/G)\to H^n(X)$ must be nontrivial. Next, we observe that $p^*:H^n(X/G)\to H^n(X)$ can not be onto. Let if possible, then $H^n(X/G)\cong\mathbb{Z}_2\oplus\mathbb{Z}_2$. Therefore, there exist nonzero element $u,v\in H^n(X/G)$ such that $p^*(u )=x$ and $p^*(v)=y$. So,  $p^*(uv )=xy\neq 0 $ in $H^{2n}(X)$. Consequently,  $0\neq uv \in H^{2n}(X/G)$ which contradicts  Lemma \ref{L1S13}. Therefore, $\im (p^*:H^n(X/G)\to H^n(X))\cong\mathbb{Z}_2,$  and hence  $H^n(X/G)\cong\mathbb{Z}_2$. Let $\{v\}$ be the basis of $H^n(X/G)$. Then,   $p^*(v)=x$ or $y$. As the sequence
	$
	0\rightarrow \ker \rho\rightarrow H^n(X)\stackrel{\rho}{\rightarrow}\im \rho \rightarrow 0 
	$
	is split exact, we get $\im \rho \cong\mathbb{Z}_2$. This gives that $H^{n+1}(X/G)=0$. Accordingly, we get $
	 H^n(X/G)\cong H^{2n-3}(X/G)$  and $H^{2n-j}(X/G)=0$ for $j=1,2$ and 4. Thus, we have
	\begin{align*}
	H^i(X/G)=\begin{cases}
	\Z_2 &\text{ if } 0\leq i\equiv 0\text{(mod 4)}<n,  n\leq i\equiv 3\text{(mod 4)}<2n\\
	0 &\text{ otherwise}
	\end{cases}
	\end{align*} 
	As $H^{n+1}(X/G)=H^{2n}(X/G)=0$, we have $u^{\frac{n+1}{4}}=0=v^2$.
	 Therefore, $H^*(X/G)$ is given by
	${\mathbb{Z}_2[u,v]}/{\langle u^\frac{n+1}{4},v^2\rangle} $
	where $\deg u=4,\deg v=n$. This realizes possibility(iii).

		\noindent\textbf{Case(III)}: 	As $n\equiv $0(mod 4), $H^{n-4}(X/G)\cong\Z_2$ with basis $\{u^\frac{n-4}{4}\}$ and $H^{n-j}(X/G)=0$ for $1\leq j\leq 3$.
		This implies that $H^{n+j}(X/G)=0$ for all $1\leq j\leq 3$ and   $p^*:H^n(X/G)\to  H^{n}(X)$ is surjective. Now, the  sequence
	$
		0\to \im \cup \to H^n(X/G)\stackrel{p^*}{\rightarrow}H^n(X)\to 0
		$
		is split exact, therefore, 
		$H^n(X/G)\cong \mathbb{Z}_2\oplus\mathbb{Z}_2$ 
		with basis $\{u^{\frac{n}{4}},v\}$ where $p^*(v)=x$. From this we get, $H^i(X/G)=0$ for $n<i\equiv j $(mod 4)$<m$ for $1\leq j\leq 3$ and $H^i(X/G)\cong \Z_2\oplus\Z_2$ for $n\leq i\equiv 0 $(mod 4)$<m$ with basis $\{vu^{\frac{i-n}{4}},u^{\frac{i}{4}}\}$. As both $n$ and $m$ can not be even, there are two possibilities for $m$: (a) $m\equiv$3(mod 4), and $m\equiv$1(mod 4)

		\noindent \textbf{Subcase(a):}		
		As $m\equiv 3$(mod 4), we get $H^{m-j}(X/G)=0$ for $j=1,2$ or 4 and $H^{m-3}(X/G)\\ \cong \Z_2\oplus\Z_2$. By Lemma \ref{L2S3},  $p^*:H^m(X/G)\to H^m(X)$ must be trivial. 
		By the exactness of the Gysin sequence, $H^{m+j}(X/G)=0$  for $j=0,2,3$ and $\rho:H^{m}(X)\to H^{m-3}(X/G)$ is injective. This gives that $\ker(\cup:H^{m-3}(X/G)\to H^{m+1}(X/G))\cong\mathbb{Z}_2$ and hence  $H^{m+1}(X/G)\cong \mathbb{Z}_2$. Consequently,  $H^{n+m-j}(X/G)=0$ for $j=1,2,4$ and $H^{n+m-3}(X/G)\cong H^{m+1}(X/G)$. Thus, we have
		\begin{align*}
		H^i(X/G)=\begin{cases}
		\mathbb{Z}_2 &\text{ if } 0\leq i\equiv 0\text{(mod 4)} <n,\,m< i\equiv 0\text{(mod 4)}<n+m\\
		\mathbb{Z}_2\oplus\mathbb{Z}_2 &\text{ if }n\leq i\equiv0\text{(mod 4)}<m\\
		0 &\text{ otherwise}
		\end{cases}
		\end{align*}
		Note  that a basis for $H^{m-3}(X/G)$ is $\{vu^{\frac{m-n-3}{4}},u^{\frac{m-3}{4}}\}$. As $H^{m+1}(X/G)\cong\mathbb{Z}_2$, one basis element vanish under the map $\cup:H^{m-3}(X/G)\to H^{m+1}(X/G)$.

		 \noindent So, there are two  cases: ($vu^{\frac{m-n+1}{4}}=0$ \& $u^{\frac{m+1}{4}}\neq 0$) or ($vu^{\frac{m-n+1}{4}}\neq 0$ \& $u^{\frac{m+1}{4}}= 0$).

		 If $vu^{\frac{m-n+1}{4}}=0$ and $u^{\frac{m+1}{4}}\neq 0$ then the basis of $H^{m+1}(X/G)$ is $\{u^{\frac{m+1}{4}}\}$, and hence basis of $H^{n+m-3}(X/G)$ is $\{u^{\frac{n+m-3}{4}}\}$. It is clear that  $u^{\frac{n+m+1}{4}}=0$. As $2n\equiv 0$(mod 4) and $v^2\in H^{2n}(X/G)$,  we have $v^2=\alpha v u^{\frac{n}{4}}+\beta u^{\frac{n}{2}}$ for $\alpha,\beta\in\mathbb{Z}_2$ and $\alpha=0$ if $2n>m$. Therefore, the graded algebra of $X/G$ is given by
		$$H^*(X/G)=\frac{\mathbb{Z}_2[u,v]}{\langle u^\frac{n+m+1}{4},vu^{\frac{m-n+1}{4}},v^2+\alpha v u^{\frac{n}{4}}+\beta u^{\frac{n}{2}}\rangle} $$
		where $\deg u=4,\deg v=n$, $\alpha,\beta\in\mathbb{Z}_2$ and $\alpha=0$ if $2n>m$.
		This realizes possibility (ii) of the theorem.

		 If $vu^{\frac{m-n+1}{4}}\neq 0$ and $u^{\frac{m+1}{4}}= 0$ then the   basis for $H^{m+1}(X/G)$ is $\{vu^{\frac{m-n+1}{4}}\}$ and hence the basis for $H^{n+m-3}(X/G)$ is $\{vu^{\frac{m-3}{4}}\}$. Obviously, $vu^{\frac{m+1}{4}}=0$. As above, we get  $v^2=\alpha v u^{\frac{n}{4}}+\beta u^{\frac{n}{2}}$ for $\alpha,\beta\in\mathbb{Z}_2$ and $\beta=0$ if $2n>m$. Hence, we have
		\begin{align*}
		H^*(X/G)=\frac{\mathbb{Z}_2[u,v]}{\langle u^\frac{m+1}{4},v^2+\alpha v u^{\frac{n}{4}}+\beta u^{\frac{n}{2}}\rangle}
		\end{align*}
		where $\deg u=4,\deg v=n$, $\alpha,\beta\in\mathbb{Z}_2$ and $\beta=0$ if $2n>m$. This realizes possibility(i) of the theorem.

		\noindent \textbf{Subcase(b):} As  $m\equiv 1\text{(mod 4)} $, we get  $H^{m-1}(X/G)\cong\mathbb{Z}_2\oplus\mathbb{Z}_2$, and $H^{m-j}(X/G)=0$ for $j=2,3$ and 4. By the exactness of the Gysin sequence,  $p^*:H^m(X/G)\to H^m(X)$ must be nontrivial, a contradiction. For $m=n+1$,  we get the same  cohomology groups.

		\noindent\textbf{Case(IV)}: When $n\equiv$2(mod 4).

		As $n\equiv$2(mod 4), we get $H^{n-2}(X/G)\cong \mathbb{Z}_2$ with the basis $\{u^{\frac{n-2}{4}}\}$ and $H^{n-j}(X/G)=0$ for $j=1,3,4$.  By the exactness of the Gysin sequence,    $H^n(X/G)\cong  H^{n}(X)$, $H^{n+1}(X/G)=H^{n+3}(X/G)=0$ and $H^{n+2}(X/G)\cong H^{n-2}(X/G)$. The Bases for $H^n(X/G)$ and $H^{n+2}(X/G)$ are $\{v\}$ and  $\{u^{\frac{n+2}{4}}\}$, respectively, where $p^*(v)=x$.  Further,  $H^i(X/G)=0$ for $n<i\equiv$j(mod 4)$<m$, $j=1,3$; $H^i(X/G)\cong\Z_2$ for $n\leq i\equiv$2(mod 4)$<m$ with basis $\{vu^{\frac{i-n}{4}}\}$ and  $H^i(X/G)\cong\Z_2$ for $n< i\equiv$0(mod 4)$<m$ with basis $\{u^{\frac{i}{4}}\}$. By Lemma \ref{L2S3},  $p^*:H^m(X/G)\to H^m(X)$ must be trivial. Here, there are two possibilities for $m$: (a) $m\equiv$3(mod 4); and (b) $m\equiv$1(mod 4).

		\noindent \textbf{Subcase(a):} As $m\equiv$3(mod 4), we get $H^{m-3}(X/G)\cong H^{m-1}(X/G)\cong\Z_2$ and  $H^{m-2}(X/G)\\=H^{m-4}(X/G)=0$. Consequently,  $H^{m+3}(X/G)\cong H^{m-1}(X/G)$ and $H^{m+j}(X/G)=0$ for all $0\leq j\leq 2$, and hence  $H^{n+m-j}(X/G)=0$ for $j=1,2,4$ and $H^{n+m-3}(X/G)\cong\Z_2$. Thus, we have
		\begin{align*}
		H^i(X/G)=\begin{cases}
		\mathbb{Z}_2 &\text{ if } 0\leq i\equiv 0\text{(mod 4)}<m,\,n\leq i\equiv 2\text{(mod 4)}<n+m\\
		0 &\text{ otherwise}
		\end{cases}
		\end{align*}
It is clear that $u^{\frac{m+1}{4}}=0$.	As  $2n\equiv 0$(mod 4), we get $v^2=\alpha u^{\frac{n}{2}}$ for $\alpha\in\mathbb{Z}_2$ and $\alpha=0$ if $2n>m$. Therefore,
		$H^*(X/G)={\mathbb{Z}_2[u,v]}/{\langle u^\frac{m+1}{4},v^2+\alpha u^{\frac{n}{2}}\rangle} $
		where $\deg u=4$ and $\deg v=n$. This realizes possibility (i).
For $m=n+1$, we get the same result.

		\noindent \textbf{Subcase(b):} 
		As $m\equiv$1(mod 4), we get $ H^{m-j}(X/G)\cong\Z_2$ for $j=1,3$; and 0 for $j=2,4$. This gives that $H^m(X/G)=H^{m+1}(X/G)=0$. By the exactness of the Gysin sequence, we get $H^{m+3}(X/G)\cong H^{m-1}(X/G)$, $H^{m+2}(X/G)=0$. Consequently,  $H^{m+n-j}(X/G)=0$ for $j=1,2,4$ and $H^{m+n-3}(X/G)\cong\Z_2$. Thus, we have  
		\begin{align*}
		H^i(X/G)=\begin{cases}
		\mathbb{Z}_2 &\text{ if } 0\leq i\equiv 0\text{(mod 4)}<n+m,\,n\leq i\equiv 2\text{(mod 4)}<m\\
		0 &\text{ otherwise}
		\end{cases}
		\end{align*}
			It is clear that $u^{\frac{n+m+1}{4}}=vu^{\frac{m-n-3}{4}}=0$.
		As  $2n\equiv 0$(mod 4), we get $v^2=\alpha u^{\frac{n}{2}}$ for $\alpha\in\mathbb{Z}_2$. Therefore, the graded cohomology algebra of $X/G$  is 
		$$H^*(X/G)=\frac{\mathbb{Z}_2[u,v]}{\langle u^\frac{n+m+1}{4},vu^{\frac{m-n+1}{4}},v^2+\alpha u^{\frac{n}{2}}\rangle} $$
		where $\deg u=4,\deg v=n$, $\alpha\in\mathbb{Z}_2$ and this realizes possibility (ii). For $m=n+3$, we get the same result.     \qedhere

			\end{proof}
\end{theorem}

Similarly, we get the cohomological structure of the orbit spaces of free $G=\mathbb{S}^1$ actions on a finitistic space $X\sim _2 \S^n\times \S^m$. For example, $G=\S^1$ acts freely on  $SU(3)\sim_2 \S^3\times\S^5$ by the component-wise multiplication but $SU(3)\not\approx \S^3\times\S^5$. Recall that the same cohomology algebra with rational coefficients is determined in \cite{Dotzel}. 

\begin{theorem}\label{TS1}
	Let $X$ be a finitistic free $G$-space with $X\sim_2\S^n\times \S^m,1\leq n\leq m$, where $G=\S^1$. Then $H^*(X/G)$ is isomorphic to one of the following graded algebras:
	\begin{enumerate}
		\item[(i)] $\mathbb{Z}_2[u, v ]/\langle u^{\frac{m+1}{2}}, v^2+\alpha u^n+\beta v u^\frac{n}{2}\rangle$, where $\deg u=2$, $\deg v=n$, $m$ is odd,  $\alpha=0$ if  $m\leq 2n$ and $\beta=0$ if $n$ is odd.

		\item[(ii)] $\mathbb{Z}_2[u,v]/\langle u^{\frac{n+m+1}{2}}, vu^{\frac{m-n+1}{2}}, v^2+\alpha u^n +\beta vu^{\frac{n}{2}} \rangle$, where $\deg u=2$, $\deg v=n$, $m-n$ is odd and $\beta=0 $ if $m<2n$ or $m$ is even.
		
		\item[(iii)] $\mathbb{Z}_2[u, v]/\langle u^{\frac{n+1}{2}}, v^2\rangle$, where $\deg u=2$, $\deg v=m$ and $n$ is  odd.
	\end{enumerate}
	\end{theorem}

From the above theorems, we observe that
\begin{remark}
	If $G=\S^d$, $d=1$ or 3, acts freely on a finitistic space $X\sim_2\S^n\times \S^m,1\leq n\leq m$, then $n\equiv d$(mod $(d+1)$) or $m\equiv d$(mod $(d+1)$) or $m-n\equiv d$(mod $(d+1)$).
\end{remark}

Using the above  cohomological calculations,  we get the mod 2 cohomology index of $X$ and obtain Borsuk-Ulam type results for $G=\S^d, d=1$  or 3, actions on a finitistic space $X\sim_2\S^n\times\S^m,1\leq n\leq m. $
\begin{corollary}
		Let $X$ be a finitistic free $G$-space with $X\sim_2\S^n\times \S^m,1\leq n\leq m$, where $G=\S^d$, $d=1$ or 3. Then, cohom-index$_GX$ is
	\begin{itemize}
		\item[(i)] $\frac{m-d}{d+1}$ if $m\equiv d$(mod d+1)
		\item[(ii)] $\frac{n+m-d}{d+1}$ if $(m-n)\equiv d$(mod d+1)
		\item[(iii)]   $\frac{n-d}{d+1}$ if $n\equiv d$(mod d+1).
	\end{itemize}
\begin{proof}
	By Theorem \ref{TS3} and \ref{TS1}, $u\in H^{d+1}(X/G)$ is the Steifel-Witney class of the sphere bundle $G\hookrightarrow X\to X/G$. The result follows from the definition of cohom-index$_G$.   
\end{proof}
\end{corollary}
Next, we obtain Borsuk-Ulam type results and proof follows from corollary \ref{ind<coind}.
\begin{corollary}
	Let $X$ be a finitistic free $G$-space with $X\sim_2\S^n\times \S^m,1\leq n\leq m$, where $G=\S^d$, $d=1$ or 3. Then, there is no $G$-equivariant $f:\S^{(d+1)k+d}\to X$
\begin{itemize}
	\item[(i)] for   $k\geq\frac{m+1}{d+1}$, if cohom-index$_G=\frac{m-d}{d+1}$
	\item[(ii)] for   $k\geq\frac{n+m+1}{d+1}$, if cohom-index$_G=\frac{n+m-d}{d+1}$
	\item[(iii)] for   $k\geq\frac{n+1}{d+1}$, if cohom-index$_G=\frac{n-d}{d+1}$
\end{itemize}
\end{corollary}
\section{Examples}
\noindent \textbf{5.1.}
	Let $G=\S^{d}$, $d=1$ or 3, acts on  $\S^{(d+1)n+d}$ by the standard action and trivially  on $\S^m$. Then $G$  acts freely on $X=\S^{(d+1)n+d}\times \S^m$ by the diagonal action. Note that $X/G=\mathbb{FP}^n\times \S^m$, where $\mathbb{F}=\C$ or $\mathbb{H}$ where $d=1$ or $3$.  So, $H^*(X/G)=\Z_2[u,v]/\langle u^{n+1}, v^2\rangle, \text{ where } \deg u=d+1$ and $\deg v=m$. This realizes case(i)\&(iii) of Theorem \ref{TS3} and \ref{TS1}. 	 So, cohom-index$_G X=n$. By above corollary, there is no $G$-equivariant map $f:\S^{(d+1)k+d}\to X$ for $k\geq n+1$. As $g:\S^{(d+1)n+d}\to X$ defined as $g(x)=(x,y_0)$  is $G$-equivariant map where $y_0\in \S^m$ be any point. Thus, ind$_G X=n$.

\noindent \textbf{5.2.}
	Let $(a_0,a_1,\cdots, a_n)$ and $(b_0,b_1,\cdots,b_m)$ be sequences of integers such that $\gcd(a_i,b_j)=1$ for all $0\leq i\leq n$ and $0\leq j\leq m$. Then, $G=\S^1$ acts freely on $X=\S^{2n+1}\times\S^{2m+1}$ by
	$\lambda.(z,w)=((\lambda^{a_0}z_0,\lambda^{a_1}z_1,\cdots,\lambda^{a_n}z_n),(\lambda^{b_0}w_0,\lambda^{b_1}w_1,\cdots,\lambda^{b_m}w_m))$
	where $z=(z_0,z_1,\cdots,z_n)$, $w=(w_0,w_1,\cdots,w_m)$ and $z_i, w_i\in \C$. If $n=0$ then $X/G=L_{a_0}^{2m+1}(b_0,b_1,$\\
	$,\cdots,b_m)$. So, for $a_0$ odd,  $X/G\sim_2 \mathbb{S}^{2m+1}$; for $p$  even but $4\not| p$, $X/G\sim_2 \mathbb{RP}^{2m+1}$; and for $4| p$,    $X/G\sim_2 \S^1\times\mathbb{CP}^m.$  This realizes case(i) of theorem \ref{TS1}.

\bibliographystyle{amsplain}

\end{document}